\documentclass[a4paper]{article}
\usepackage{amsthm,amsfonts,amsmath,amssymb,units}
\usepackage[abbrev,nobysame]{amsrefs}
\usepackage[cp1251]{inputenc}
\usepackage[english]{babel}
\usepackage[final]{graphicx}
\usepackage{setspace}
\usepackage[12pt]{extsizes}
\begin{document}
\newcommand{\per}{{\rm per}}
\newcommand{\supp}{{\rm supp}}
\newtheorem{teorema}{Theorem}
\newtheorem{lemma}{Lemma}
\newtheorem{utv}{Proposition}
\newtheorem{svoistvo}{Property}
\newtheorem{sled}{Corollary}
\newtheorem{con}{Conjecture}
\newtheorem{zam}{Remark}
\newtheorem{quest}{Question}
\newtheorem{claim}{Claim}

\author{A. A. Taranenko}
\title{Enumeration and constructions of vertices of the polytope of polystochastic matrices}
\date{}

\maketitle

\begin{abstract}
A multidimensional nonnegative matrix is called polystochastic if the sum of entries in each of its lines equals $1$. The set of all polystochastic matrices of order $n$ and dimension $d$ is a convex polytope $\Omega_n^d$ known as the Birkhoff polytope.  In this paper, we identify all vertices of the polytopes  $\Omega_4^3$ and  $\Omega_3^4$ correcting the results of Ke, Li, and Xiao (2016). Additionally, we  describe constructions vertices of $\Omega_n^d$ using multidimensional matrix  products and  find symmetric vertices of $\Omega_3^d$ for all $d \geq 4$ with large support sizes.

\textbf{Keywords:} polystochastic matrix, Birkhoff polytope, vertex, symmetric matrix, permanent

\textbf{MSC2020:} 15B51, 52B05, 15A15

\end{abstract}

\section{Introduction}

A \textit{$d$-dimensional matrix $A$ of order $n$} is an array $(a_\alpha)_{\alpha \in I^d_n}$, $a_\alpha \in\mathbb R$, its entries are indexed by $\alpha$ from the index set $I_n^d = \{  \alpha = (\alpha_1, \ldots, \alpha_d ) | \alpha_i \in \{ 0, \ldots, n-1\}\}$. The \textit{Hamming distance} $\rho(\alpha, \beta)$ between $\alpha, \beta \in I_n^d$ is the number of positions in which they differ.  A matrix $A$ is called \textit{nonnegative} if all $a_{\alpha} \geq 0$, and it is a \textit{$(0,1)$-matrix} if all its entries are $0$ or $1$.  The \textit{support} $\supp(A)$ of a matrix $A$ is the set of all indices $\alpha$ such that $a_{\alpha} \neq 0$. Denote the  cardinality of the support of $A$ by $N(A)$.

Given $k\in \left\{0,\ldots,d\right\}$, a \textit{$k$-dimensional plane} in $A$ is the submatrix  obtained by fixing $d-k$ positions in  the indices and letting  the values in the other $k$ positions vary from $1$ to $n$. We will say that the set of fixed positions defines the \textit{direction} of a plane, and planes with the same direction are \textit{parallel}.   A $1$-dimensional plane is said to be a \textit{line}, and a $(d-1)$-dimensional plane is a \textit{hyperplane}. Note that each  $d$-dimensional matrix of order $n$ has $dn$ hyperplanes and  $dn^{d-1}$ lines.  

Matrices $A$ and $B$ are called \textit{equivalent} if one can be obtained from the other by permuting  hyperplanes of the same direction and/or  permuting the  directions of hyperplanes.

For a $d$-dimensional matrix $A$ of order $n$ and an index $\alpha \in I_{n}^d$, let $A_{\alpha}$ denote the $d$-dimensional matrix of order $n-1$  obtained  by deleting all hyperplanes containing the index $\alpha$. 

A multidimensional nonnegative matrix $A$ is called \textit{polystochastic} if the sum of its entries in each line equals $1$. Polystochastic matrices of dimension $2$ are known as \textit{doubly stochastic} matrices.  Since doubly stochastic $(0,1)$-matrices are exactly the permutation matrices,  for  $d \geq 3$  we   call  $d$-dimensional polystochastic $(0,1)$-matrices \textit{multidimensional permutations}.

A matrix $A$ is a  \textit{convex combination} of  matrices $B_1, \ldots, B_k$ if there exist coefficients  $\lambda_1, \ldots, \lambda_k$, $0 < \lambda_i < 1$, $\lambda_1 + \cdots + \lambda_k = 1$ such that 
$$A = \lambda_1 B_1 + \cdots + \lambda_k B_k.$$ 

Let $\Omega_n^d$ denote the set of  all $d$-dimensional polystochastic matrices of order $n$. Note that  any convex combination of matrices from  $\Omega_n^d$  is itself  a $d$-dimensional polystochastic matrix of order $n$. So the set $\Omega_n^d$  is a convex polytope called the \textit{Birkhoff polytope}.  

A matrix $A \in \Omega_n^d$ is a \textit{vertex} of the Birkhoff polytope $\Omega_n^d$ if,  for every convex combination $A = \lambda_1 B_1 + \cdots + \lambda_k B_k$,  all $B_i$ are identical to $A$. In other words, vertices  of  the  polytope $\Omega_n^d$ cannot be expressed as  nontrivial convex combinations of its  elements. The definition implies that every  $d$-dimensional polystochastic matrix of order $n$  is a convex combination of vertices of $\Omega_n^d$, and every $d$-dimensional permutation of order $n$ is a vertex of $\Omega_n^d$.  

The problem of determining all vertices of the Birkhoff polytope  was first raised by Jurkat and Ryser in~\cite{JurRys.stochmatr}.

The well-known Birkhoff theorem states that every vertex of the polytope of doubly stochastic matrices is a permutation matrix.  

It is not hard to prove (see, for example,~\cite{my.obzor})  that the polytope $\Omega_2^d$ of $d$-dimensional polystochastic matrices of order $2$ has dimension $1$,  and  all vertices of $\Omega_2^d$ are multidimensional permutations. 

The smallest example of a vertex  that is not  a multidimensional permutation  is the following $3$-dimensional polystochastic matrix of order $3$

$$
V = \left( \begin{array}{ccc|ccc|ccc}
1 & 0 & 0 & 0 & \nicefrac{1}{2} & \nicefrac{1}{2} & 0 & \nicefrac{1}{2} & \nicefrac{1}{2} \\
0 & \nicefrac{1}{2} & \nicefrac{1}{2} & \nicefrac{1}{2} & \nicefrac{1}{2} & 0 & \nicefrac{1}{2} & 0 & \nicefrac{1}{2} \\
0 & \nicefrac{1}{2} & \nicefrac{1}{2} & \nicefrac{1}{2} & 0 & \nicefrac{1}{2} & \nicefrac{1}{2} & \nicefrac{1}{2} & 0 \\
\end{array}
\right).
$$

This vertex of $\Omega_3^3$ was discovered in papers~\cite{ChanPakZha.polystoch,csima.multmatr, CuiLiNg.Birkfortensor,LinLur.birvert, FichSwart.3dimstoch, KeLiXiao.extrpoints, WangZhang.permtensor}, among others. A computer search conducted  in~\cite{FichSwart.3dimstoch} and~\cite{ ChanPakZha.polystoch}  showed that there are only two nonequivalent vertices of the $\Omega_3^3$: the $3$-dimensional permutation of order $3$ and the matrix $V$. 

For other Birkhoff polytopes  $\Omega_n^d$, the sets of their vertices are only partially known. Although Ke, Li, and Xiao in~\cite{KeLiXiao.extrpoints}   claimed to have found all (35) nonequivalent vertices of the polytope $\Omega_4^3$, their enumeration method does not  list all vertices of this polytope. In the present paper, we show that the polytope  $\Omega_4^3$ has 533 nonequivalent vertices, most of which were missed in the previous enumeration. 

Typically,  constructions  of vertices of the Birkhoff polytope $\Omega_n^d$ different from multidimensional permutations yield polystochastic matrices with no more than two nonzero entries in each line. This class of matrices has attracted special attention due to its relative simplicity: to prove that such  a matrix is a vertex of $\Omega_n^d$ it suffices to show that it is not a convex combination of multidimensional permutations. This class of vertices appears, for example, in papers \cite{CuiLiNg.Birkfortensor, FichSwart.3dimstoch}. Moreover,  Linial and Luria  in~\cite{LinLur.birvert}    improved the lower bound on the number of vertices of   $\Omega_n^3$ by  constructing many vertices  with exactly two nonzero entries in each line.   Additional details on known lower and upper bounds on the number of vertices of $\Omega_n^d$ can be found in~\cite{PotTar.polyvertbounds}.

It is evident  that the minimum size of the support of matrices from $\Omega_n^d$ is achieved by multidimensional permutations. Fischer and Swart~\cite{FichSwart.3dimstoch} asked what  the next minimum value for the support of matrices in $\Omega_n^3$ is, but this question remains unanswered.
For $3$-dimensional polystochastic matrices, Jurkat and Ryser in~\cite{JurRys.stochmatr}  established  the upper bound on the size of the support, which  was later generalized by the author in~\cite{my.obzor}  for polystochastic matrices of any dimension and nonnegative multidimensional matrices with equal sums in $k$-dimensional planes.

\begin{utv}[\cite{JurRys.stochmatr,my.obzor}] \label{supportbound}
Let $A$ be a vertex of the polytope of $d$-dimensional polystochastic matrices  of order $n$. Then for  the size $N(A)$  of the support of $A$ we have
$$N(A) \leq n^d - (n-1)^d.$$
\end{utv}

Among  all vertices of $\Omega_n^d$, some have significant symmetries.  A  $d$-dimensional matrix of order $n$ is called  \textit{symmetric} if for every permutation $\sigma \in S_d$ we  have $a_{\alpha} = a_{\sigma(\alpha)}$, where $\sigma(\alpha) = (\alpha_{\sigma(1)}, \ldots, \alpha_{\sigma(d)})$.  In other words, a matrix $A$ is symmetric if any transpose does not change the matrix.  
 The set of $d$-dimensional symmetric polystochastic matrices of order $n$ is also a convex polytope, but not every vertex of the polytope of symmetric polystochastic matrices is a vertex of $\Omega_n^d$.  For example,  adjacency matrices of regular graphs that do not contain perfect matchings yield matrices different from permutations that are vertices of the polytope of symmetric doubly stochastic matrices.

Historically, the study  of the polytope of doubly stochastic matrices was motivated by  van der Waerden's conjecture  on the minimum value of the permanent (see~\cite{minc.perbook}). Vertices of the Birkhoff polytope $\Omega_n^d$ also play a crucial role in the problem of the positiveness of the permanent of polystochastic matrices.  The \textit{permanent} of a $d$-dimensional matrix $A$ of order $n$ is defined as 
$$\per A = \sum\limits_{\ell \in D(A)} \prod\limits_{\alpha \in \ell} a_{\alpha}, $$
 where $D(A)$ is the set of all diagonals of $A$, and a \textit{diagonal} in $A$ is a collection of indices $\{ \alpha^1, \ldots, \alpha^n \}$ such that no indices share the same hyperplane.

The  Birkhoff theorem  implies  that  the permanent of every doubly stochastic matrix is positive.  For odd $d \geq 3$ and even $n$, there exist examples of $d$-dimensional permutations of order $n$ (which are vertices of $\Omega_n^d$) with  zero permanents. Such permutations can be found, for example, in the survey \cite{wanless.surv} by Wanless.  For other values of $n$ and $d$, no $d$-dimensional polystochastic matrices of order $n$ with zero permanents are known.    Thus  the author proposed  the following conjecture  in~\cite{my.obzor}:

\begin{con}[\cite{my.obzor}]
Every polystochastic matrix of even dimension or odd order has a positive permanent.
\end{con}

If a  matrix $A \in \Omega_n^d$ is a convex combination  of vertices $B_1, \ldots, B_k$, then  $A$ has a zero permanent only if  all matrices $B_1, \ldots, B_k$ also have zero permanents.  Hence,  the above conjecture can be reduced to the assertion that  all vertices of $\Omega_n^d$ have a positive permanent if $d$ is even or $n$ is odd.

The main aim of the present paper is to enumerate all vertices of the Birkhoff polytope $\Omega_n^d$ for small $n$ and $d$ and provide explicit constructions of some vertices. For the constructed vertices, we pay  special attention to their permanents and the size and structure of their supports. In particular, in Section~\ref{enumersec}  we describe  the algorithm for  enumerating vertices of the polytopes $\Omega_3^4$ and $\Omega_4^3$ and provide some statistics on   sizes of their supports  and values of entries. The complete list of the vertices of $\Omega_3^4$ is given in the Appendix, while the list of vertices of $\Omega_4^3$ is available at~\cite{vert34.data}. 

In the second part of the paper, we  propose several constructions of vertices of $\Omega_n^d$ for growing $n$ or $d$. In Section~\ref{prodconstrsec}, we consider two constructions of vertices based on the multidimensional matrix multiplication.  In Section~\ref{symvertsec} we present a symmetric vertex of $\Omega_3^d$ for all $d \geq 4$. Finally, in Section~\ref{discussec}, we connect our study with the more general problem of enumerating vertices of the polytope of fractional transversals in a hypergraph.

\section{Enumeration of vertices} \label{enumersec}

We start with the necessary background for enumerating the vertices of the polytope $\Omega_n^d$. 

First, Jurkat and Ryser in~\cite{JurRys.stochmatr} prove that  there is a one-to-one correspondence between the vertices of the polytope $\Omega_n^d$ and their supports.

\begin{utv}[\cite{JurRys.stochmatr}] \label{vertbysupp}
Let $T$ be a $d$-dimensional $(0,1)$-matrix of order $n$ and  $A \in \Omega_n^d$ be such that $\supp(A) = \supp(T)$. There is a unique matrix $A$ satisfying this property if and only if $A$ is a vertex of $\Omega_n^d$.  
\end{utv}

Proposition~\ref{vertbysupp} implies  that  enumerating all vertices of $\Omega_n^d$ reduces  identifying all of their supports.  Note that  the size   $N(A)$  of the support of a vertex $A$ of $\Omega_n^d$  is at least  $n^{d-1}$ (the size of the support of a multidimensional permutation) and, by Proposition~\ref{supportbound},  is not greater than $n^d - (n-1)^d$.  So   the set of  all subsets   $S$  of $I_n^d$  with $n^{d-1} \leq |S| \leq n^{d} - (n-1)^d$  contains all  supports of vertices of $\Omega_n^d$. Moreover, if $S \subseteq I_n^d$ is a support of some $d$-dimensional  polystochastic matrix of order $n$, then it satisfies the following necessary conditions:

(C0) every line contains at least one index from $S$;

(C1) if a line $\ell$  contains only one index $\alpha \in S$, then for every other line $\ell'$ with $\alpha \in \ell'$,  we have  $\ell' \cap S = \{ \alpha\}$. 

It is not hard to see that  these  conditions are not sufficient.  
For the $2$-dimensional  case,   the following characterization of supports  of doubly stochastic matrices is known (see, e.g., \cite{SinKnop.totalsup}, where such matrices are said to be matrices with total support).

\begin{utv} \label{totalsup}
Let $T$ be  a $2$-dimensional $(0,1)$-matrix   of order $n$. There is a doubly stochastic matrix $A$ such that $\supp(A) = \supp(T)$ if and only if for every $t_{i,j} = 1$ we have $\per T_{i,j} > 0$.
\end{utv}

\begin{proof}
The necessity follows from the Birkhoff theorem stating  that all vertices of the polytope $\Omega_n^2$ are permutation matrices. 

To prove sufficiency, note that the matrix $A$ with entries $a_{i,j} = \frac{\per T_{i,j}}{\per T}$, if $m_{i,j} = 1$, and $a_{i,j} = 0$, if $t_{i,j} = 0$, is a doubly stochastic matrix such that $\supp(A) = \supp(T)$.
\end{proof}

For the similar result on the supports of  polystochastic matrices, the complete enumeration of all vertices of $\Omega_n^d$ is required.

\begin{utv}
A  $d$-dimensional $(0,1)$-matrix $T$ of order $n$ has a support of some polystochastic  matrix $A$ if and only if for every $\alpha \in \supp (T)$ there exists a vertex $B$ of $\Omega_n^d$ with $\alpha \in \supp (B) $.
\end{utv}

\begin{proof}
The necessity follows from the definition of the polytope $\Omega_n^d$. 

Let us prove sufficiency. Consider the set $\mathcal{V} (T)$ of all vertices $B$ of $\Omega_n^d$ such that $\supp(B) \subseteq \supp (T)$. Then the matrix $A = \frac{1}{|\mathcal{V}(T)|} \sum\limits_{B \in \mathcal{V}(T)} B $ is a polystochastic matrix. Since for every $\alpha \in \supp(T)$ there exists $B \in \mathcal{V}(T)$ such that $\alpha \in \supp(B)$, we have that $\supp(A) = \supp (T)$. 
\end{proof}

We can also characterize supports of vertices of polystochastic matrices as sets of indices that  avoid  some special subsets.   A \textit{zero-sum matrix}  is defined as a  real-valued $d$-dimensional matrix $B$ of order $n$  in which  the sum of entries in each line is  zero.  If a  set $R \subseteq I_n^d$  is a support of some zero-sum matrix, we will say that $R$  is a \textit{zero-sum set}.

\begin{utv}[\cite{JurRys.stochmatr}] \label{JRtradecrit}
A $d$-dimensional polystochastic matrix $A$ of order $n$ is a vertex of $\Omega_n^d$ if and only if there are no nonempty zero-sum sets $R$ in $I_n^d$ such that $R \subseteq \supp(A)$. 
\end{utv}

Krotov and Potapov in~\cite{KrotPot.Latbitrades} studied some zero-sum sets in multidimensional matrices of order~$3$, but  their results  were not sufficient to characterize this set completely. Moreover,  the conducted computer enumeration of some zero-sum sets in matrices of small dimensions witnesses that the number of zero-sum sets  is comparable with the number of   vertices.

Thus  we need to use another method to check whether a given set of indices $S$ is a support of a polystochastic matrix and, moreover, a support of a vertex of the Birkhoff polytope.  It will be based on the incidence matrix between the lines and  indices from the support.

Let $T$ be a  $d$-dimensional $(0,1)$-matrix of order $n$ and  $N =N(T)$.  Let us enumerate the indices in $\supp(T)$ by numbers  from $1$ to $N$ and  the lines in the matrix $T$ by numbers  from $1$ to $dn^{d-1}$. Then the \textit{incidence matrix} is the rectangular matrix $L = L(T)$ of size $ dn^{d-1} \times N$ such that $\ell_{i,j} = 1$ if the $i$-th line of $T$ contains  the $j$-th element of the support, and $\ell_{i,j} = 0$ otherwise. 

From the definition of polystochastic matrices, we have that, for a given $d$-dimensional $(0,1)$-matrix $T$ of order $n$ with the incidence matrix $L$, there exists a matrix $A \in \Omega_n^d$ with $\supp(A) = \supp(T)$ if and only if equation $Lx = \textbf{1}$ has a positive solution $x  = (x_1, \ldots, x_N)$. Here $\textbf{1}$ stands for the all-one vector of length $dn^{d-1}$,  and entries $x_i$  give nonzero entries of the polystochastic matrix $A$.

Note that if $N \leq n^{d} - (n-1)^{d}$, then the incidence matrix $L$ has more rows than columns (because $n^{d} - (n-1)^{d} < dn^{d-1}$).  It means that the rank of the incidence matrix $L$ is not greater than $N$ and, moreover, if  $rank (L)) = N$, then there is no more than one solution of equation $Lx = \textbf{1}$.  So Proposition~\ref{vertbysupp} can be reformulated in the following way.

\begin{utv} \label{vertexsupp}
Let $T$ be a  $d$-dimensional $(0,1)$-matrix of order $n$ such that $N = N (T) \leq n^{d} - (n-1)^{d}$, and let $L$ be the incidence matrix for $T$. Then the support of $T$ is the support of a vertex of $\Omega_n^d$ if and  only if $rank(L) = N $ and there exists a solution $x = (x_1, \ldots, x_N)$ of $Lx = \textbf{1}$ with all $x_i > 0$. 
\end{utv}

Summing up, we are ready to describe a general enumeration strategy of  vertices of $\Omega_n^d$. 
\medskip

\textbf{Algorithm 1.}

\begin{enumerate}
\item List  $d$-dimensional $(0,1)$-matrices $T$ of order $n$ with $n^{d-1} \leq N(T) \leq n^d - (n-1)^d$ satisfying conditions (C0), (C1), and, possibly, whose support avoids some zero-sum sets.

\item For every such matrix $T$, find the size $N$ of the support and the incidence matrix $L$.

\item If $rank(L) = N$, solve the equation $Lx = \textbf{1}$.

\item If there is a feasible solution $x = (x_1, \ldots, x_N)$ of  $Lx = \textbf{1}$ such that all $x_i > 0$, then the matrix $A$ obtained from $T$ by assigning the value $x_i$ to the $i$-th element of the support of $T$ is a vertex of $\Omega_n^d$.
\end{enumerate}

In the next sections, we implement this algorithm on polytopes $\Omega_3^4$ and $\Omega_4^3$. In particular,  in these cases we describe  how to enumerate the possible supports of vertices more effectively and provide statistics for the resulting sets of vertices.

\subsection{Vertices of $\Omega_3^4$} \label{vert43sec}

By Proposition~\ref{supportbound},  for the size of the support $N(A)$ of a vertex $A$ of the polytope $\Omega_3^4$, we have that $27 \leq N(A) \leq 65$. Moreover, it is well known  that there is a unique $4$-dimensional permutation of order $3$ (a vertex of $\Omega_3^4$)  with the size of support equal  to $27$.

To narrow the set  of potential supports for vertices of  $\Omega_3^4$, we state several claims  derived from the definition of a polystochastic matrix through straightforward considerations.

\begin{claim} \label{01hyper}
Let $A \in \Omega_3^4$ be different from the multidimensional permutation. If there is a hyperplane $\Gamma$ in $A$ such that $\Gamma$ is  a $(0,1)$-matrix, then $A$ is a convex combination of two multidimensional permutations. In particular,  $A$ is not a vertex of $\Omega_3^4$.
\end{claim}

\begin{claim} \label{21indist3}
Let $A \in \Omega_3^4$ be different from the multidimensional permutation. If there are indices $\alpha, \beta \in I_3^4$ such that the Hamming distance between $\alpha$ and $\beta$ is equal to $3$, then the hyperplane $\Gamma$ containing both $\alpha$ and $\beta$ is a $(0,1)$-matrix. In particular, by Claim~$\ref{01hyper}$, $A$ is not a vertex of $\Omega_3^4$. 
\end{claim}

\begin{claim} \label{tw02dim01}
Let $A \in \Omega_3^4$ be different from the multidimensional permutation. If there is a hyperplane $\Gamma$ of $A$ containing two $2$-dimensional planes $P_1$ and $P_2$ of different directions such that $P_1$ and $P_2$ are $(0,1)$-matrices, then, by Claim~\ref{21indist3}, $\Gamma$ is a $(0,1)$-matrix. In particular, by Claim~$\ref{01hyper}$,  $A$ is not a vertex of $\Omega_3^4$. 
\end{claim}

\begin{claim} \label{2ones2dim}
Let $A \in \Omega_3^4$ be different from the multidimensional permutation. If some $2$-dimensional plane $P$ of $A$ contains two $1$s, then $P$ is a $(0,1)$-matrix.
\end{claim}

\begin{claim} \label{22diminpoint}
Let $A \in \Omega_3^4$ be different from the multidimensional permutation. Assume that there are  $2$-dimensional   $(0,1)$-planes  $P_1$ and $P_2$ in $A$ that intersect at exactly one entry $a_{\alpha}$.   
\begin{itemize}
\item If $a_{\alpha} = 0$, then there is a $2$-dimensional $(0,1)$-plane $P_3$  intersecting $P_1$ or $P_2$ along a line. In this case, by Claim~$\ref{tw02dim01}$, $A$ is not a vertex.
\item If $a_{\alpha} = 1$, then such a matrix $A$ is unique, and $A$ is equivalent to vertex \#$2$ (see Appendix).
\end{itemize}
\end{claim}

Claims~\ref{01hyper}--\ref{22diminpoint} imply that if $A$ is a vertex of  $\Omega_3^4$  such that $A$ has an entry equal to $1$  and $A$ is different from the multidimensional permutation (vertex \#1)  or vertex \#2 , then all indices of  $1$-entries of $A$   are either located in some $2$-dimensional plane or located at the (maximum possible) distance  $4$ from each other. 

To find other possible supports $T$  for  vertices  of $\Omega_3^4$, we present every $4$-dimensional $(0,1)$-matrix $T$  of order $3$   as a union of $9$ parallel $2$-dimensional planes  $P_{1}, \ldots, P_{9}$:
$$
T = \left(
\begin{array}{c|c|c}
P_{1} & P_{2} & P_{3} \\
\hline
P_{4} & P_{5} & P_{6} \\
\hline
P_{7} & P_{8} & P_{9} 
\end{array}
\right).$$

Without loss of generality, we may assume that $N(P_{1}) \leq N(P_{i})$ for all $i \in \{ 1,\ldots, 9\}$.  Since   $N(T) \leq 65$, we have $N(P_{1}) \leq 7$. Thus,  the following possibilities exist for $P_{1}$  as the  support of a doubly stochastic matrix:
$$
\mathcal{P}_1 =
\begin{array}{ccc}
 1 & 0 & 0 \\ 0 & 1 & 0 \\ 0 & 0 & 1 
\end{array}; ~~
\mathcal{P}_2 =
\begin{array}{ccc}
1 & 0 & 0 \\ 0 & 1 & 1 \\ 0 & 1 & 1 
\end{array}; ~~
\mathcal{P}_3 =
\begin{array}{ccc}
1 & 1 & 0 \\ 0 & 1 & 1 \\ 1 & 0 & 1 
\end{array}; ~~
\mathcal{P}_4 =
\begin{array}{ccc}
0 & 1 & 1 \\ 1 & 0 & 1 \\ 1 & 1 & 1 
\end{array}.
$$

By Claims~\ref{01hyper}--\ref{22diminpoint}, no other plane $P_{i}$ is equivalent to $\mathcal{P}_1$. Thus, for $i \in \{ 2, \ldots, 9\}$, each plane $P_i$ is equivalent to one of the following  planes with the support of a  doubly stochastic matrix:
$$
\mathcal{P}_2 =
\begin{array}{ccc}
1 & 0 & 0 \\ 0 & 1 & 1 \\ 0 & 1 & 1 
\end{array}; ~~
\mathcal{P}_3 =
\begin{array}{ccc}
1 & 1 & 0 \\ 0 & 1 & 1 \\ 1 & 0 & 1 
\end{array}; ~~
\mathcal{P}_4 =
\begin{array}{ccc}
0 & 1 & 1 \\ 1 & 0 & 1 \\ 1 & 1 & 1 
\end{array};~~ 
\mathcal{P}_5 =
\begin{array}{ccc}
0 & 1 & 1 \\ 1 & 1 & 1 \\ 1 & 1 & 1 
\end{array}; ~~
\mathcal{P}_6 =
\begin{array}{ccc}
1 & 1 & 1 \\ 1 & 1 & 1 \\ 1 & 1 & 1 
\end{array}.
$$
There are $9$ matrices equivalent to $\mathcal{P}_2$, $6$ matrices equivalent to $\mathcal{P}_3$,  $18$ matrices equivalent to $\mathcal{P}_4$, and $9$ matrices equivalent to $\mathcal{P}_5$, making a total  of no more than $43$ possibilities for each plane  $\mathcal{P}_i$,$i \in \{ 2, \ldots, 9\}$. Moreover,   Claims~\ref{01hyper}--\ref{22diminpoint}  imply that if $P_1  = \mathcal{P}_1$, then none of $P_i$, $i \in \{ 2, \ldots, 9\}$, is equivalent to $\mathcal{P}_2$, and if $P_1  = \mathcal{P}_2$, then none of  planes $P_2$, $P_3$, $P_4$, and $P_7$ is equivalent to $\mathcal{P}_2$. 

Taking these restrictions into account, we  enumerate  all  collections of planes  $P_1, \ldots, P_9$ of the matrix $T$ such that $N(P_1) + \cdots + N(P_9) \leq 65$ and the  matrix $T$ satisfies conditions (C0) and (C1). To shorten computations,  conditions (C0) and (C1)   can   be   during the construction of $T$ , as each new plane $P_i$ completes a $3$-dimensional plane of $T$.  

To find all vertices of the polytope  $\Omega_3^4$,  we perform steps 2--4 of Algorithm 1  for  each matrix $T$.  We establish that there are  $21$ equivalence classes of vertices $\Omega_3^4$. In Appendix, one can find  a representative for each  equivalence class, along with the value of their permanent and the size of their support. We also note that there are $6$ symmetric vertices of $\Omega_3^4$: \#1, \#4, \#7, \#9, \#13, \#19.

The distribution of sizes of the supports of vertices $A \in \Omega_3^4 $ is given in Table~1. In Table~2 we present the  distribution of  the least common multiple of  the denominators of the entries of the vertices, i.e., the minimal integer positive values $\Delta$ such  that $\Delta \cdot  A$ is an integer matrix.
\begin{center}
\begin{tabular}{c||c|c|c|c|c|c|c|c|c|c|c|c}
$N(A)$ & 27 & 49 & 51 & 52 & 58 & 59 & 60 & 61 &  62 & 63 & 64 & 65 \\
\hline
\# of $A$ & 1 & 1 & 1 & 1 & 1 & 2 & 2 & 2 &  3 & 5 &  1 & 1 
\end{tabular}

Table 1. Sizes of supports of vertices of $\Omega_3^4$. 
\end{center}

\begin{center}
\begin{tabular}{c||c|c|c|c|c|c}
$\Delta$ & 1 & 2 & 3 & 4 & 5 & 6  \\
\hline
\# of $A$ & 1 & 3 & 6 & 6 & 4 & 1 
\end{tabular}

Table 2. LCM of denominators of entries of vertices $A$ of $\Omega_3^4$. 
\end{center}

\subsection{Vertices of $\Omega_4^3$} \label{vert34sec}

By Proposition~\ref{supportbound},  for the size of the support $N(A)$ of a vertex $A$ of the polytope $\Omega_4^3$, we have that $16 \leq N(A) \leq 37$. 

To find possible supports $T$  for  vertices  of $\Omega_4^3$, we present every $3$-dimensional $(0,1)$-matrix $T$  of order $4$   as a union of $4$ parallel $2$-dimensional planes  $P_{1}, \ldots, P_{4}$:

$$
T = \left(
\begin{array}{c|c|c|c}
P_{1} & P_{2} & P_{3} & P_4 \\
\end{array}
\right).$$

Without loss of generality, we may assume that $N(P_{1}) \leq N(P_{2}) \leq  N(P_{3})  \leq  N(P_{4})$.   Since   $N(T) \leq 37$, we have $N(P_{1}) \leq 9$. So, by   Proposition~\ref{totalsup}, for the plane  $P_{1}$ we have the following possibilities:
\begin{gather*}
\begin{array}{cccc}
1 & 0 & 0 & 0  \\ 0 & 1 & 0 & 0 \\ 0 & 0 & 1  & 0 \\ 0 & 0 & 0  & 1
\end{array}; ~~
\begin{array}{cccc}
1 & 0 & 0 & 0  \\ 0 & 1 & 0 & 0 \\ 0 & 0 & 1  & 1 \\ 0 & 0 & 1  & 1
\end{array}; ~~
\begin{array}{cccc}
1 & 0 & 0 & 0  \\ 0 & 1 & 1 & 0 \\ 0 & 0 & 1  & 1 \\ 0 & 1 & 0  & 1
\end{array}; ~~
\begin{array}{cccc}
1 & 0 & 0 & 0  \\ 0 & 1 & 1 & 0 \\ 0 & 0 & 1  & 1 \\ 0 & 1 & 1  & 1
\end{array}; ~~
\begin{array}{cccc}
1 & 1 & 0 & 0  \\ 1 & 1 & 0 & 0 \\ 0 & 0 & 1  & 1 \\ 0 & 0 & 1  & 1
\end{array}; ~~
\begin{array}{cccc}
1 & 1 & 0 & 0  \\ 0 & 1 & 1 & 0 \\ 0 & 0 & 1  & 1 \\ 1 & 0 & 0  & 1
\end{array}; ~~
\\
\begin{array}{cccc}
1 & 0 & 0 & 0  \\ 0 & 1 & 1 & 0 \\ 0 & 1 & 1  & 1 \\ 0 & 1 & 1  & 1
\end{array}; ~~
\begin{array}{cccc}
0 & 1 & 1 & 1  \\ 1 & 1 & 0 & 0 \\ 1 & 0 & 1  & 0 \\ 1 & 0 & 0  & 1
\end{array}; ~~
\begin{array}{cccc}
1 & 1 & 1 & 0  \\ 0 & 0 & 1 & 1 \\ 1 & 0 & 0  & 1 \\ 1 & 1 & 0  & 0
\end{array}.
\end{gather*}

The remaining planes $P_2$, $P_3$, and $P_4$ may have support of any doubly stochastic matrix of order $4$, which can be enumerated  through  exhaustive search using Proposition~\ref{totalsup}  (7443 supports in total,  considering supports of equivalent matrices as distinct). 

Giving these possibilities for planes $P_1$, $P_2$, $P_3$, and $P_4$, we construct $3$-dimensional $(0,1)$-matrices $T$ with possible supports by sequentially adding all feasible planes $P_{i+1}$ to a given sequence of planes $P_1, \ldots, P_i$. To eliminate matrices $T$ that cannot have a support of a vertex of $\Omega_4^3$,  at each step we  check condition (C1) and that the support of this partial matrix is not excessively large. Also,  we construct the  incidence matrix $L$ between the lines and the support of the partial matrix and confirm that the rank of $L$ equals the size of the support. If this condition is not met, the support of the partial matrix contains a zero-sum set.  Finally, after completing the  construction of the full matrix $T$, we check that all other $2$-dimensional planes have support of doubly stochastic matrices (using Proposition~\ref{totalsup}). 

To identify all vertices of the polytope  $\Omega_4^3$,  we perform steps 2--4 of Algorithm 1  for  each resulting matrix $T$. We establish that there are  $533$ equivalence classes of vertices $\Omega_4^3$. A list of their representatives  can be found in~\cite{vert34.data}. For each vertex, we compute the size of support,  permanent,  and the order of its automorphism group.  

Our analysis reveals that there are $11$ symmetric vertices of $\Omega_4^3$ corresponding to the lines 
 1, 2, 6, 11, 27, 120,  142, 143, 146,  485,  and 527
  in  list~\cite{vert34.data}.

The distribution of support sizes for vertices $A \in \Omega_4^3 $ is presented in Table 3. Table~4 provides the  distribution of  the least common multiple of the  denominators of the entries  in the  vertices, i.e., the minimal integer positive values $\Delta$ such  that $\Delta \cdot  A$ is an integer matrix.
\begin{center}
\begin{tabular}{c||c|c|c|c|c|c|c|c|c|c|c|c|c}
$N(A)$ & 16 & 25 & 27 & 28 & 29 & 30 & 31 & 32 &  33 & 34 & 35 & 36 & 37 \\
\hline
\# of $A$ & 2 & 1 & 2 & 1 & 2 & 7 & 7 & 16 &  22 & 31 &  103 & 103 & 236 
\end{tabular}

Table 3. Sizes of supports of vertices of $\Omega_4^3$. 
\end{center}

\begin{center}
\begin{tabular}{c||c|c|c|c|c|c|c|c}
$\Delta$ & 1 & 2 & 3 & 4 & 5 & 6 & 7 & 8  \\
\hline
\# of $A$ & 2 & 18 & 155 & 215 & 104 & 32 & 5 & 2 
\end{tabular}

Table 4. LCM  of denominators of entries of vertices $A$ of $\Omega_4^3$. 
\end{center}

\section{Product construction of vertices} \label{prodconstrsec}

In this section, we propose constructions of vertices of the Birkhoff polytope $\Omega_n^d$ based on the Kronecker and dot products of multidimensional matrices. 

Let $A$ be a $d$-dimensional matrix of order $n_1$ and $B$ be a $d$-dimensional  matrix of order $n_2$.  The \textit{Kronecker product} $A\otimes B$ of matrices $A$ and $B$ is the $d$-dimensional matrix $C$ of order $n_1 n_2$ with entries $c_{\gamma} = a_\alpha b_\beta,$
where  $\gamma_i = \alpha_i  n_2 + \beta_i$ for each $i = 1, \ldots, d$. 

In~\cite{my.permofproduct} it was proved that the Kronecker of polystochastic matrices is a polystochastic matrix. We show that the Kronecker product of a vertex of $\Omega_n^d$ and a multidimensional permutation is also a vertex.

\begin{teorema} \label{Kronecvert}
If $A$ is a $d$-dimensional permutation of order $n_1$, and $B$ is a vertex of $\Omega^d_{n_2}$, then $A \otimes B$ is a vertex of the polytope $\Omega^d_{n_1 n_2}$ with  $N(A \times B) = n_1^d \cdot N(B)$.
\end{teorema}

\begin{proof}
Since $A$ is a permutation matrix, the definition of the Kronecker product implies that the matrix  $C = A \otimes B$ is composed of $n_1^d$  submatrices (blocks) of order $n_2$,  where each block is either the matrix $B$ or the  zero matrix of order $n_2$. Indeed, if  $a_{\alpha} = 1$, then  $c_{\gamma} = b_\beta$  for all indices $\gamma$  with $\gamma_i = \alpha_i  n_2 + \beta_i$,  $i = 1, \ldots, d$, and if $a_{\alpha} = 1$, then  $c_{\gamma} = 0$ for all indices $\beta$ and $\gamma$ such that    $\gamma_i = \alpha_i  n_2 + \beta_i$,  $i = 1, \ldots, d$. Since $A$ is a multidimensional permutation, each line in $C$ intersects exactly one block $B$.

If $C = A \otimes B$ is not a vertex of $\Omega_{n_1 n_2}^d$, then there exists a nontrivial   convex combination $C = \lambda D_1  + (1 -  \lambda) D_2$, where  $D_1$ and $D_2 $ are polystochastic matrices different from $C$.  Moreover, matrices $D_1$  and $D_2$ also have a block structure similar to the matrix $C$.  Then there is a block $B$ in the matrix $C$ for which we have a nontrivial convex combination $B = \lambda D'_1  + ( 1 - \lambda) D'_2$, where $D'_1$ and $D'_2$ are  the corresponding  blocks of matrices $D_1$ and $D_2$. Note that matrices $D'_1$ and $D'_2$ are polystochastic, because $C$ is a polystochastic matrix and $B$ is a polystochastic block of $C$. Thus, we obtain a nontrivial convex combination for the matrix $B$ in $\Omega_{n_2}^d$: a contradiction.
\end{proof}

\begin{zam}
The Kronecker product of two arbitrary vertices of the Birkhoff polytope is not necessarily to be a vertex. For example, for the  vertex $V \in \Omega_3^3$ different from the multidimensional permutation, the product $V \otimes V$ is not a vertex of $\Omega^3_9$. 
\end{zam}

\begin{zam}
In Theorem~\ref{Kronecvert}, we can replace each block $B$ in the Kronecker product $A \otimes B$ by other vertices of $\Omega_{n_2}^d$.  The same reasoning gives that  such a matrix is also a vertex of $\Omega_{n_1 n_2}^d$. 
\end{zam}

 Let $A$  be a $d_1$-dimensional matrix of order $n$  and $B$ be  a $d_2$-dimensional matrix of order $n$. We define the \textit{dot product} $A \cdot B$ of matrices $A$ and $B$ to be the $(d_1 + d_2 - 2)$-dimensional matrix  $C$  such  that $c_\gamma = \sum\limits_{i=1}^n a_{\alpha, i} b_{i, \beta}$, where $\gamma = \alpha \beta$ is the concatenation of indices $\alpha \in I_n^{d_1 -1}$ and $\beta \in I_n^{d_2 -1}$.
 
In~\cite{my.permofproduct} it was proved that the dot product of polystochastic matrices is a polystochastic matrix. To prove  that the dot product of a vertex of $\Omega_n^d$ and a multidimensional permutation is  a vertex, we need the following lemma.

 \begin{lemma} \label{planeofdot}
Let  $A$ be a $d_1$-dimensional permutation of order $n$ and $B$ be a $d_2$-dimensional matrix of order $n$. Then all $d_2$-dimensional planes  of  $A \cdot B$  of  direction, in which the first $d_1 - 2$ positions of indices are fixed and the last $d_2$ positions  vary, are  matrices equivalent to the matrix $B$.  
 \end{lemma}
 
 \begin{proof}
 The proof is by induction on $d_1$. For $d_1 = 2$, the matrix $A$ is a  permutation matrix, and the dot product $A \cdot B$ is a $d_2$-dimensional matrix equivalent to the matrix in which hyperplanes of  the first direction are permuted according to the permutation $A$.
 
 Assume that $d_1 > 2$. Let us present the matrix $A$ as a union of $n$ parallel hyperplanes $\Gamma_1, \ldots, \Gamma_n$ of the first direction. Since $A$ is a multidimensional permutation, every hyperplane $\Gamma_i$ of $A$ is also a multidimensional permutation.  The definition of the dot product implies that the  hyperplanes of the first direction of $A \cdot B$ are equal to $\Gamma_1 \cdot B, \ldots, \Gamma_n \cdot B$. By the inductive assumption, for each $i = 1, \ldots, n$, all $d_2$-dimensional planes of $\Gamma_i \cdots B$  of direction,  in which the first $d_1 - 3$ positions of indices are fixed and the last $d_2$ positions  vary,  are equivalent to the matrix $B$. Then the similar statement holds for the matrix $A \cdot B$. 
 \end{proof}

\begin{teorema} \label{dotconstruction}
If $A$ is a $d_1$-dimensional permutation of order $n$, and $B$ is a  vertex of $\Omega^{d_2}_{n}$, then $A \cdot B$ is a vertex of the polytope $\Omega^{d_1 + d_2 -2}_{n}$ with  $N(A \cdot  B) = n^{d_1 -2} \cdot N(B)$.
\end{teorema}

\begin{proof}
If $C = A \cdot B$ is not a vertex of $\Omega^{d_1 + d_2 -2}_{n}$, then there exists a nontrivial   convex combination $C = \lambda D_1  + (1 -\lambda) D_2$, where  $D_1$ and $D_2 $ are polystochastic matrices that are different from the matrix $C$.  By Lemma~\ref{planeofdot}, all  $d_2$-dimensional planes $P$  of  $C$ of some direction are equivalent to the matrix $B$.   Let  us choose  the plane $P$ so that  there is a  nontrivial convex combination $P = \lambda D'_1  + (1 -\lambda) D'_2$, where matrices $D'_1$ and $D'_2$ are polystochastic and different from the plane $P$.  Then there is a nontrivial combination for the matrix $B$: a contradiction.
\end{proof}

\begin{zam}
The dot product of two arbitrary vertices of the Birkhoff polytope is not necessarily to be a vertex. For example, for the vertex $V \in \Omega_3^3$ different from the multidimensional permutation,  the product $V \cdot V$ is not a vertex of $\Omega^4_3$. 
\end{zam}

At last, we give one more way to construct vertices of the Birkhoff polytope $\Omega_n^d$.

\begin{lemma}
Let $A$ be a $d$-dimensional polystochastic matrix of order $n$ and $\Gamma_1, \ldots, \Gamma_n$ be its hyperplanes of some direction. If $\Gamma_1, \ldots, \Gamma_{n-1}$ are vertices of $\Omega_n^{d-1}$, then $A$ is a vertex of $\Omega_n^d$. 
\end{lemma}

\begin{proof}

Since hyperplanes $\Gamma_1, \ldots, \Gamma_{n-1}$  of $A$ are vertices of $\Omega_n^{d-1}$, they do not have nontrivial   convex combinations. In particular, for every convex combination  $A = \lambda B_1  + (1 - \lambda) B_2$, the  polystochastic matrices $B_1$ and $B_2$ coincide in these $n-1$ hyperplanes. Since every $d$-dimensional polystochastic  matrix of order $n$ can be uniquely restored by parallel $n-1$ hyperplanes, matrices $B_1$ and $B_2$ are equal. Therefore, every convex combination for the matrix $A$ is trivial, and $A$ is a vertex of $\Omega_n^d$. 
\end{proof}

\section{Symmetric vertex of $\Omega_3^d$} \label{symvertsec}

The aim of the present section  is to  construct    symmetric vertices of $\Omega_3^d$  for all $d \geq 4$ generalizing the vertex~\#7 of the polytope $\Omega_3^4$. The  entries of these vertices will be equal to  $0, \nicefrac{1}{3}, \nicefrac{2}{3}$, or $1$, and, as far as we know, they give  the first nontrivial series of vertices of $\Omega_n^d$ when $n$ is fixed ($n = 3$) and $d \rightarrow \infty$. 

Our construction requires the following lemma stating that every  zero-sum matrix  of order $n$ is  uniquely defined by the values in a submatrix of order $n-1$.

\begin{lemma} \label{minorrecover}
Let $A$ be a $d$-dimensional zero-sum matrix of order $n$.  Let $\alpha, \beta \in I_n^d$  be such that    $\rho(\alpha, \beta) = d-k$ and $\Gamma(\alpha, \beta)$ be the $k$-dimensional plane in the submatrix  $A_{\beta}$ composed of all indices $\gamma$ satisfying  $\gamma_i = \alpha_i$ for all $i $ such that $\alpha_i \neq \beta_i$. Then  
$$a_\alpha =   (-1)^k \sum\limits_{\gamma \in \Gamma(\alpha,\beta)} a_{\gamma} .$$
\end{lemma}

\begin{proof}

We prove the lemma by induction on $k$. 

For $k = 0$ it is nothing to prove. If $k = 1$, then index $\alpha$ coincides with    $\beta$  in exactly one component. In this case,  $a_\alpha$  is the following sum in the $1$-dimensional plane $\Gamma(\alpha,\beta)$ of the submatrix $A_\beta$:
$$a_{\alpha} = - \sum\limits_{\gamma \in \Gamma(\alpha,\beta)} a_{\gamma}.$$

Let us prove the induction step. Suppose that $\alpha$ is an index such that $\rho(\alpha, \beta) = d - k$.  Without loss of generality, assume that $\alpha$ and $\beta$ coincide in the first $k$ components: $\alpha_i = \beta_i$ for all $i = 1, \ldots, k$.  Let $\alpha^1, \ldots, \alpha^{n-1}$  denote all  indices that coincide with $\alpha$  in all components, except for the $k$-th component. Then  $\rho (\alpha^j, \beta) = d - k +1$ for all $j =1, \ldots, n-1$, and since $A$ is a zero-sum matrix, we have
$$a_{\alpha} = - \sum\limits_{j=1}^{n-1} a_{\alpha^j}.$$
By the induction hypothesis, 
$$a_{\alpha^j} =   (-1)^{k-1} \sum\limits_{\gamma \in \Gamma(\alpha^j,\beta)} a_{\gamma}. $$
By the definition, the  union of all planes  $ \Gamma(\alpha^j,\beta)$, $j = 1, \ldots, n-1$ is the plane $\Gamma(\alpha,\beta)$. Thus we obtain
$$
a_{\alpha} =  - \sum\limits_{j=1}^{{n-1}}  \left(  (-1)^{k-1} \sum\limits_{\gamma \in \Gamma(\alpha^j,\beta)} a_{\gamma} \right)  = 
 (-1)^k \sum\limits_{\gamma \in \Gamma(\alpha,\beta)} a_{\gamma} .
$$
\end{proof}

In what follows, we consider only $d$-dimensional matrices $A$ of order $3$ and index sets $I_3^d$. For a given index $\alpha \in I_3^d$, we denote by $|\alpha|_0, |\alpha|_1$, and $|\alpha|_2$ the number of $0$s, $1$s, and $2$s in $\alpha$, respectively.

\textbf{Construction 1.}

Let $d$ be even. Define  the entries of the symmetric $d$-dimensional matrix $A$ of order $3$ as
$$
a_{\alpha} = \begin{cases}
\nicefrac{1}{3}, \mbox{ if all }  |\alpha|_0, |\alpha|_1, |\alpha|_2 \neq 0; \\
0, \mbox{ if }  |\alpha|_0 = 0 \mbox{ and }  |\alpha|_1, |\alpha|_2 \mbox{ are even and } \neq d;\\
\nicefrac{2}{3}, \mbox{ if }  |\alpha|_0 = 0  \mbox{ and } |\alpha|_1, |\alpha|_2 \mbox{ are odd}; \\
\nicefrac{2}{3}, \mbox{ if }  |\alpha|_1 = 0 \mbox{ and }  |\alpha|_0, |\alpha|_2 \mbox{ are even and } \neq d;\\
0, \mbox{ if }  |\alpha|_1 = 0  \mbox{ and } |\alpha|_0, |\alpha|_2 \mbox{ are odd}; \\
\nicefrac{2}{3}, \mbox{ if }  |\alpha|_2 = 0 \mbox{ and }  |\alpha|_0, |\alpha|_1 \mbox{ are even and } \neq d;\\
0, \mbox{ if }  |\alpha|_2 = 0  \mbox{ and } |\alpha|_0, |\alpha|_1 \mbox{ are odd}; \\
\nicefrac{1}{3}, \mbox{ if }  |\alpha|_1 \mbox{ or } |\alpha|_2 = d; \\
1, \mbox{ if }  |\alpha|_0 = d. \\
\end{cases}
$$

For odd $d$, we define the entries of the symmetric $d$-dimensional matrix $A$ of order $3$ as
$$
a_{\alpha} = \begin{cases}
\nicefrac{1}{3}, \mbox{ if all }  |\alpha|_0, |\alpha|_1, |\alpha|_2 \neq 0; \\
0, \mbox{ if }  |\alpha|_0 = 0,  |\alpha|_1 \neq d  \mbox{ is odd, and } |\alpha|_2 \mbox{ is even};\\
\nicefrac{2}{3}, \mbox{ if }  |\alpha|_0 = 0,  |\alpha|_1  \mbox{ is even, and } |\alpha|_2 \neq d \mbox{ is odd};\\
\nicefrac{2}{3}, \mbox{ if }  |\alpha|_1 = 0,  |\alpha|_0 \neq d  \mbox{ is odd, and } |\alpha|_2 \mbox{ is even};\\
0, \mbox{ if }  |\alpha|_1 = 0,  |\alpha|_0  \mbox{ is even, and } |\alpha|_2 \neq d \mbox{ is odd};\\
0, \mbox{ if }  |\alpha|_2 = 0,  |\alpha|_0 \neq d  \mbox{ is odd, and } |\alpha|_1 \mbox{ is even};\\
\nicefrac{2}{3}, \mbox{ if }  |\alpha|_2 = 0,  |\alpha|_0  \mbox{ is even, and } |\alpha|_1 \neq d \mbox{ is odd};\\
\nicefrac{1}{3}, \mbox{ if } |\alpha|_0, |\alpha|_1 \mbox{ or } |\alpha|_2 = d. \\
\end{cases}
$$

\begin{utv}
The $d$-dimensional matrix $A$ of order $3$ from Construction $1$ is a polystochastic matrix. The size of its support is $N(A) = 3^d -  3 \cdot 2^{d-1} +2 $ if $d$ is even, and  $N(A) = 3^d -  3 \cdot 2^{d-1} +3 $ if $d$ is odd.
\end{utv}

\begin{proof}
To prove that $A$ is a polystochastic matrix, it is sufficient to check that the sum of the entries in each line is equal to $1$.  

If $d$ is even,  have the following types of lines $(\alpha, \beta, \gamma)$ in $A$ and values $a_\alpha$, $a_\beta$, and $a_\gamma$:
\begin{center}
\begin{tabular}{c|c}
$|\alpha|_0, |\alpha|_1, |\alpha|_2 \neq 0$  & $\nicefrac{1}{3}$ \\
\hline
$|\beta|_0, |\beta|_1, |\beta|_2 \neq 0$ &  $\nicefrac{1}{3}$ \\
\hline
$|\gamma|_0, |\gamma|_1, |\gamma|_2 \neq 0$& $\nicefrac{1}{3}$ \\
\end{tabular}~~~~
\begin{tabular}{c|c}
$|\alpha|_0 = 1$, $|\alpha|_1 \neq 0$ is even, $|\alpha|_2$ is odd  & $\nicefrac{1}{3}$ \\
\hline
$|\beta|_0 = 0$, $|\beta|_1$, $|\beta|_2 $  are odd &  $\nicefrac{2}{3}$ \\
\hline
$|\gamma|_0 = 0$, $|\gamma|_1, |\gamma|_2 \neq 0$ are even & $0$ \\
\end{tabular}

\medskip
\begin{tabular}{c|c}
$|\alpha|_1 = 1$, $|\alpha|_0  \neq 0$ is even, $|\alpha|_2$ is odd & $\nicefrac{1}{3}$ \\
\hline
$|\beta|_1 = 0$, $|\beta|_0, |\beta|_2 $ are odd &  $0$ \\
\hline
$|\gamma|_1 = 0$, $|\gamma|_0, |\gamma|_2 \neq 0$ are even  & $\nicefrac{2}{3}$ \\
\end{tabular}~~~~
\begin{tabular}{c|c}
$|\alpha|_2 = 1$, $|\alpha|_0 \neq 0$ is even,  $|\alpha|_1 $ is odd & $\nicefrac{1}{3}$ \\
\hline
$|\beta|_2 = 0$, $|\beta|_0$, $|\beta|_1 $ are odd &  $0$ \\
\hline
$|\gamma|_2 = 0$, $|\gamma|_0, |\gamma|_1 \neq 0$ are even & $\nicefrac{2}{3}$ \\
\end{tabular}

\medskip
\begin{tabular}{c|c}
$|\alpha|_0 = d$ & $1$ \\
\hline
$|\beta|_0 = d-1$, $|\beta|_1 = 1$  &  $0$ \\
\hline
$|\gamma|_0 = d-1$, $|\gamma|_2 =1$  & $0$ \\
\end{tabular}~~~~
\begin{tabular}{c|c}
$|\alpha|_1 = d$ & $\nicefrac{1}{3}$ \\
\hline
$|\beta|_1 = d-1$, $|\beta|_0 = 1$  &  $0$ \\
\hline
$|\gamma|_1 = d-1$, $|\gamma|_2 =1$  & $\nicefrac{2}{3}$ \\
\end{tabular}~~~~
\begin{tabular}{c|c}
$|\alpha|_2 = d$ & $\nicefrac{1}{3}$ \\
\hline
$|\beta|_2 = d-1$, $|\beta|_0 = 1$  &  $0$  \\
\hline
$|\gamma|_2 = d-1$, $|\gamma|_1 =1$  & $\nicefrac{2}{3}$  \\
\end{tabular}
\end{center}

If $d$ is odd, then there are the  following types of lines $(\alpha, \beta, \gamma)$ in $A$ and values $a_\alpha$, $a_\beta$, and $a_\gamma$:
\begin{center}
\begin{tabular}{c|c}
$|\alpha|_0, |\alpha|_1, |\alpha|_2 \neq 0$  & $\nicefrac{1}{3}$ \\
\hline
$|\beta|_0, |\beta|_1, |\beta|_2 \neq 0$ &  $\nicefrac{1}{3}$ \\
\hline
$|\gamma|_0, |\gamma|_1, |\gamma|_2 \neq 0$& $\nicefrac{1}{3}$ \\
\end{tabular}~~~~
\begin{tabular}{c|c}
$|\alpha|_0 = 1$, $|\alpha|_1,  |\alpha|_2 \neq 0$ are even & $\nicefrac{1}{3}$ \\
\hline
$|\beta|_0 = 0$, $|\beta|_1$ is odd, $|\beta|_2 \neq 0$ is even &  $0$ \\
\hline
$|\gamma|_0 = 0$, $|\gamma|_1\neq 0$ is even, $|\gamma|_2$ is odd & $\nicefrac{2}{3}$ \\
\end{tabular}

\medskip
\begin{tabular}{c|c}
$|\alpha|_1 = 1$, $|\alpha|_0,  |\alpha|_2 \neq 0$ are even & $\nicefrac{1}{3}$ \\
\hline
$|\beta|_1 = 0$, $|\beta|_0$ is odd, $|\beta|_2 \neq 0$ is even &  $\nicefrac{2}{3}$ \\
\hline
$|\gamma|_1 = 0$, $|\gamma|_0\neq 0$ is even, $|\gamma|_2$ is odd & $0$ \\
\end{tabular}~~~~
\begin{tabular}{c|c}
$|\alpha|_2 = 1$, $|\alpha|_0,  |\alpha|_1 \neq 0$ are even & $\nicefrac{1}{3}$ \\
\hline
$|\beta|_2 = 0$, $|\beta|_0$ is odd, $|\beta|_1 \neq 0$ is even &  $0$ \\
\hline
$|\gamma|_2 = 0$, $|\gamma|_0\neq 0$ is even, $|\gamma|_1$ is odd & $\nicefrac{2}{3}$ \\
\end{tabular}

\medskip
\begin{tabular}{c|c}
$|\alpha|_0 = d$ & $\nicefrac{1}{3}$ \\
\hline
$|\beta|_0 = d-1$, $|\beta|_1 = 1$  &  $\nicefrac{2}{3}$ \\
\hline
$|\gamma|_0 = d-1$, $|\gamma|_2 =1$  & $0$ \\
\end{tabular}~~~~
\begin{tabular}{c|c}
$|\alpha|_1 = d$ & $\nicefrac{1}{3}$ \\
\hline
$|\beta|_1 = d-1$, $|\beta|_0 = 1$  &  $0$ \\
\hline
$|\gamma|_1 = d-1$, $|\gamma|_2 =1$  & $\nicefrac{2}{3}$ \\
\end{tabular}~~~~
\begin{tabular}{c|c}
$|\alpha|_2 = d$ & $\nicefrac{1}{3}$ \\
\hline
$|\beta|_2 = d-1$, $|\beta|_0 = 1$  &  $\nicefrac{2}{3}$  \\
\hline
$|\gamma|_2 = d-1$, $|\gamma|_1 =1$  & $0$  \\
\end{tabular}
\end{center}

To  count the size of support of matrices $A$, we use the well-known identities 
$$\sum\limits_{\begin{array}{c} 0 \leq k \leq d, \\ k \mbox{ is even} \end{array}} {d \choose k} = 2^{d-1}~~~~ \mbox{ and }  ~~~~\sum\limits_{\begin{array}{c} 0 \leq k \leq d, \\   k \mbox{ is odd} \end{array}} {d \choose k} = 2^{d-1}.$$

They imply that for all  different  $i,j \in \{ 0,1,2\}$, the number of indices $\alpha \in I_3^d$  such that $|\alpha|_i = 0$ and $|\alpha|_j$ is even (or odd) is equal to $2^{d-1}$. Counting the number of zeroes in the matrix $A$, we see that for even $d$
$$N(A) = 3^{d} -( 2^{d-1} - 2 + 2^{d-1} + 2^{d-1}) = 3^{d} - 3 \cdot 2^{d-1} +2, $$
and for odd $d$
$$N(A) = 3^{d} -( 2^{d-1} - 1 + 2^{d-1} -1 + 2^{d-1} -1) = 3^{d} - 3 \cdot 2^{d-1} +3. $$
\end{proof}

\begin{utv}
For all $d \geq 4$  the $d$-dimensional matrix $A$ of order $3$ given by Construction~$1$ is a vertex of $\Omega_3^d$.
\end{utv}

\begin{proof}
By the definition, for $d = 2$ the matrix $A$ from Construction $1$ is
$$A = \left( \begin{array}{ccc}
1 & 0 & 0 \\
0 & \nicefrac{1}{3} & \nicefrac{2}{3} \\
0 & \nicefrac{2}{3} & \nicefrac{1}{3} \\
\end{array}
\right).$$
This matrix is not a permutation matrix, and so it is not a vertex of $\Omega_3^2$.

For $d = 3$  Construction $1$ gives matrix
$$A = \left( \begin{array}{ccc|ccc|ccc}
\nicefrac{1}{3}  &  \nicefrac{2}{3}  & 0  &  \nicefrac{2}{3} & 0  &  \nicefrac{1}{3}  & 0 &  \nicefrac{1}{3} &  \nicefrac{2}{3}   \\
\nicefrac{2}{3}  & 0  & \nicefrac{1}{3} & 0  &  \nicefrac{1}{3} &  \nicefrac{2}{3}  &  \nicefrac{1}{3}  &  \nicefrac{2}{3}  & 0 \\
0 & \nicefrac{1}{3} & \nicefrac{2}{3} &  \nicefrac{1}{3} &  \nicefrac{2}{3} & 0 &  \nicefrac{2}{3}  & 0 & \nicefrac{1}{3}   \\
\end{array}
\right).$$
This matrix  is a convex combination of two multidimensional permutations, and so it is not a vertex of $\Omega_3^3$.

Assume that $d \geq 4$ is even.  By Proposition~\ref{JRtradecrit}, if there is  a non-zero zero-sum  matrix $B$  such that $\supp(B) \subseteq \supp(A)$, then $A$ is  not a vertex.

Suppose that $B$ is a $d$-dimensional zero-sum matrix of order $3$ such that $\supp(B) \subseteq supp(A)$.
By Lemma~\ref{minorrecover}, all entries of the matrix $B$ can be recovered from values of some submatrix of order $2$, e.g., from  the submatrix $B_{2\cdots 2}$. 

For shortness, let $x_{r,s,t}$ denote the value $b_{\alpha}$, where $\alpha$ are such that $|\alpha|_0 = r$, $|\alpha|_1 = s$, and $|\alpha|_2 = t$.   Although in general $x_{r,s,t}$ can take several values, we show that the matrix $B$ is symmetric, and so $x_{r,s,t}$ are well defined.  By Lemma~\ref{minorrecover}, it is sufficient to show that all $x_{d-k, k, 0}$, $0\leq k \leq d$, are well defined.

From the definition of the matrix $A$ and $\supp(B) \subseteq supp(A)$, we have that $x_{d-k, k, 0} = 0$ for all odd $k$, $1 \leq k \leq d-1$. Let us show that for even $k$, $0 \leq k \leq d$, all $x_{d-k, k, 0}$ are well defined. Since $x_{0,d,0}$ and $x_{d, 0, 0}$ are values $b_{1 \cdots 1}$  and  $b_{0 \cdots 0}$, respectively,  they are well defined.  We prove by induction on $k$ that for all  even $k$, $0 < k < d$, all $x_{d-k, k, 0}$ are multiples of  $x_{0,d,0}$, and so they are well defined.   

 Since $\supp(B) \subseteq \supp(A)$, by  Lemma~\ref{minorrecover}, for all indices $\alpha$ such that $|\alpha|_0 = 0$, $|\alpha|_1 = d-k$ is even,   $|\alpha|_2 = k$ is even, $|\alpha|_1 , |\alpha|_2 \neq d$, it holds
\begin{equation} \label{minoreqvsp}
0 = b_\alpha =  (-1)^k \sum\limits_{\gamma \in \Gamma(\alpha,2\cdots 2)} b_{\gamma},
\end{equation}
where $\Gamma(\alpha,2\cdots 2)$ is the $k$-dimensional plane in the submatrix  $B_{2 \cdots 2}$ composed of all indices $\gamma$ satisfying  $\gamma_i = \alpha_i$ for all $i $ such that $\alpha_i \neq 2$.

Let  $\ell$ be even, $0 <\ell < d$, and suppose that  for every even $\ell'$, $0 \leq \ell' <\ell <d$, we proved that $x_{\ell', d- \ell', 0} $ is a multiple of $x_{0,d,0}$.  
From equation (\ref{minoreqvsp}), for every index $\alpha$ with $|\alpha|_0 = 0$, $|\alpha|_1 = d-\ell$, $|\alpha|_2 = \ell$ and the index $\beta$, obtained from $\alpha$ by changing all $2$s to $0$s  ($|\beta|_0 = \ell$, $|\beta|_1 = d-\ell$, $|\beta|_2 = 0$) we deduce 
$$ 0 = b_\alpha= - b_\beta - \sum\limits_{j = 0}^{\ell-1} {\ell \choose j} x_{j, d-j, 0} .$$ 
By the assumption, all $x_{j, d-j, 0}$ for $0 \leq j < \ell$ are zero or a multiple of $x_{0, d, 0}$. Therefore,  all  such $b_\beta$  are the same multiple of $x_{0, d, 0}$, and  all $x_{k, d-kl, 0}$ are well defined. 
 
Let us find the values  $x_{k, d-k, 0}$ now. 
Since $\supp(B) \subseteq \supp(A)$,  for all indices $\alpha$ such that $|\alpha|_1 = 0$,  $|\alpha|_0  = d-k$ is odd, $|\alpha|_2 = k$ is odd we have
\begin{equation} \label{minoreq}
0 = b_\alpha =  (-1)^k \sum\limits_{\gamma \in \Gamma(\alpha,2\cdots 2)} b_{\gamma},  
\end{equation}
where $\Gamma(\alpha,2\cdots 2)$ is the $k$-dimensional plane in the submatrix  $B_{2 \cdots 2}$ composed of all indices $\gamma$ satisfying  $\gamma_i = \alpha_i$ for all $i $ such that $\alpha_i \neq 2$.

If $k = 1$, then we have the equality   $ 0 = x_{d-1, 0, 1}  = - x_{d-1, 1, 0} - x_{d, 0 , 0}$.  Since $\supp(B) \subseteq \supp(A)$, it holds that $x_{d-1, 1, 0}  = 0$, and so  $x_{d, 0, 0}  = 0$.  

Given even $\ell$, $0 <\ell <d$, suppose that for every even $\ell'$, $0 \leq \ell' <\ell <d$, we proved that $x_{d-\ell', \ell', 0} = 0$. 
From equation~(\ref{minoreq}), we have
$$ 0 = x_{d - \ell -1, 0, \ell+ 1} = - \sum\limits_{j = 0}^{\ell +1} {\ell +1 \choose j} x_{d - j, j, 0}.$$ 
We noted before that for all odd $j$ it holds $x_{d - j, j, 0} = 0$. By the inductive assumption, for all even $j < \ell$ we have $x_{d - j, j, 0} = 0$. Thus equation~(\ref{minoreq}) means that $(\ell +1) x_{d-\ell, \ell, 0} = 0$, and so  $x_{d-\ell, \ell, 0} = 0$.

Therefore, we proved that in the submatrix  $B_{2\cdots 2}$ all entries are equal to $0$, except, possibly, $b_{1 \cdots 1} = x_{0,d,0}$. Using again Lemma~\ref{minorrecover} and the definition of the matrix $B$, we express
$$0 = x_{0, d-2, 2} = x_{2, d-2, 0} + 2 x_{1, d-1, 0} + x_{0, d, 0}.$$
Since we proved before that $x_{2, d-2, 0} = 0$ and $2 x_{1, d-1, 0} = 0$, we deduce that $x_{0, d, 0} = 0$ and thus the submatrix $B_{2 \cdots 2}$ has only zero entries. 
Application of Lemma~\ref{minorrecover} to the zero submatrix $B_{2 \cdots 2}$ gives that every zero-sum matrix $B$ such that $\supp(B) \subseteq \supp(A)$ has the empty support. Therefore, the matrix  $A$ is a vertex of $\Omega_n^d$.

Suppose now that $d \geq 5$ is odd.  Let us show  that  all zero-sum matrices $B$  such that $\supp(B) \subseteq \supp(A)$ have the empty support.  Again, with the help of  Lemma~\ref{minorrecover}  we will recover all entries of the matrix $B$ from the  submatrix $B_{2\cdots 2}$. 

As before, let  $x_{r,s,t}$ denote the value $b_{\alpha}$, where $\alpha$ is such that $|\alpha|_0 = r$, $|\alpha|_1 = s$, and $|\alpha|_2 = t$. We show later that all  $x_{r,s,t}$  are well defined.  From the definition of the matrix $A$ we have that $x_{d-k, k, 0} = 0$ for all even $k$, $k \neq 0$. 

Since $\supp(B) \subseteq \supp(A)$,   for all indices $\alpha$ such that $|\alpha|_1 = 0$, $|\alpha|_0 = d-k$ is even,  $|\alpha|_2 = k$, $k \neq d$ is odd, Lemma~\ref{minorrecover} implies that
\begin{equation} \label{minoreq2}
0 = b_\alpha =  (-1)^{k} \sum\limits_{\gamma \in \Gamma(\alpha,2\cdots 2)} b_{\gamma},
\end{equation}
where $\Gamma(\alpha,2\cdots 2)$ is  the $k$-dimensional plane in the submatrix  $B_{2 \cdots 2}$ composed by all indices $\gamma$ satisfying  $\gamma_i = \alpha_i$ for all $i $ such that $\alpha_i \neq 2$.

Acting inductively from $k = 1$  and using that  $x_{d-\ell, \ell, 0} = 0$ for all even $\ell \neq 0$, we   resolve this equation with respect to indices $\beta$ that obtained from $\alpha$ by replacing all $2$s by $1$s.  Then  for all odd $k \neq  d$  we obtain
$$x_{d-k, k, 0} = - \sum\limits_{i = 1}^{(k-1)/2} {k \choose 2i-1} x_{d-2i+1, 2i-1, 0} - x_{d, 0, 0}.$$

In particular, we proved that  all  $x_{d-k, k, 0}$  for  odd $k$,  $k \neq  d$,  are well defined and  are multiples of $x_{d, 0,0}$. In more detail,   $x_{d-2m+1, 2m-1, 0} = c_m x_{d, 0,0}$, where the sequence $c_m$ satisfies the recurrence 
$$ c_m = -1 - \sum\limits_{i = 1}^{m-1} {2m -1  \choose 2i -1} c_i. $$

The solution of this recurrence is $c_m = (-1)^m t_m$, where $t_m$ are  tangent or ``zag'' numbers (sequence A000182 in~\cite{oeis}), and  the sequence $t_m$ is increasing. Therefore, if $x_{d, 0, 0} \neq 0$, then  $|x_{d-k+2, k-2, 0}| < |x_{d-k, k, 0}|$ for all odd $k$, $3 \leq k < d$.

Using again Lemma~\ref{minorrecover} and taking into account that $x_{d-k, k, 0} = 0$ for all even  $k \neq 0$, we find that
\begin{gather*}
0 = x_{0, d-2, 2} =  x_{2, d-2, 0} + x_{0, d, 0}, \\
0 = x_{0, d-4, 4} =  x_{4,d-4, 0} + 6 x_{2,d-2, 0} + x_{0,d,0}.
\end{gather*}
Consequently, $x_{2, d-2, 0} = - x_{0, d, 0} $ and $x_{4, d-4, 0} = 5 x_{0, d, 0} $. Therefore, if $x_{0,d,0} \neq 0$, then $|x_{4, d-4, 0}| > |x_{2, d-2, 0}|$ which contradicts  the obtained property of the sequence  $x_{d-k, k, 0}$. 

Thus $x_{d, 0, 0} = x_{0,d,0} =0$. Since  $x_{d-k, k, 0}$ are multiples of $x_{d,0,0}$ for all odd  $k < d$, we proved that all entries of the submatrix $B_{2 \cdots 2}$ are equal to zero.

Application of Lemma~\ref{minorrecover} to the zero submatrix $B_{2 \cdots 2}$,  implies that every zero-sum   matrix $B$ with $\supp(B) \subseteq \supp(A)$ has the empty support. Therefore,  $A$ is a vertex of $\Omega_n^d$.
\end{proof}

\section{Discussion} \label{discussec}

While the present paper  focuses on the vertices of the polytope of polystochastic matrices, similar questions and methods are applicable for studying polytopes of fractional transversals in other hypergraphs. 

Let $H = (X, E)$ be a hypergraph with the vertex set $X$ and the set of hyperedges $E$ (where a hyperedge is some subset of vertices). A transversal  is a collection of vertices such that for every $e \in E$ there is a unique $x \in X$ belonging to $e$. A \textit{fractional transversal} is an assignment $A$ of weights $a_x \geq 0$ to vertices $x \in X$ such that    for every $e \in E$ it holds that $\sum\limits_{x \in e} a_x = 1$.  Note that a transversal can be considered as a fractional transversal where all vertices have weights $0$ or $1$

If $H$ is a $d$-uniform hypergraph (i.e., each hyperedge of $H$ consists of $d$ vertices), then  the uniform assignment $A$ (with $a_x = 1/d$ for all $x \in X$) is a fractional transversal. However, not every $d$-uniform hypergraph $H$ has a transversal. 
 
It is straightforward to observe that the set $\Omega(H)$ of all fractional transversals $A$ of a hypergraph $H$ forms a convex polytope in $\mathbb{R}^{|X|}$ defined by a system of linear inequalities  derived from the  hyperedge-vertex incidence matrix $L$ of  $H$.  Thus,  the problem of funding  vertices of  $\Omega (H)$ is a specific instance   of a  classical  linear programming problem (see, e.g.,~\cite{schrijv.thLP} and~\cite{schrijver.combopt}). 
 
By employing methods similar to those in~\cite{JurRys.stochmatr} and~\cite{my.obzor}, one can demonstrate that each vertex of $\Omega(H)$ is uniquely defined by the support and  can derive an upper bound on the size of support in terms of  the dimension of $\Omega(H)$.  Given a hypergraph $H$, a modification of Algorithm~1 can completely enumerate the vertices of $\Omega(H)$. Additionally, certain  constructions of vertices of $\Omega_n^d$ can be generalized to construct  vertices of $\Omega(H)$. 

In conclusion, we  highlight several families of hypergraphs $H$ whose transversals correspond to notable combinatorial designs:

\begin{itemize}
\item Hypergraphs $M_{k,n,d}$ with the vertex set $I_n^d$ and hyperedges equal to $k$-dimensional planes. The polytope $\Omega_n^d$ of polystochastic matrices is the polytope $\Omega (M_{1, n, d})$. Transversals  in the hypergraph $M_{k,n,d}$ are maximum distance separable (MDS) codes with distance $k+1$. Specifically, every transversal in $M_{d-1,n,d}$ is a diagonal in a $d$-dimensional matrix of order $n$. Furthermore, a matrix $A \in \Omega (M_{k, n, d}) $ with rational entries  corresponds to an orthogonal array $OA(\Delta n^{d-k},d,n,d-k)$, where $\Delta \in \mathbb{N}$ is chosen so that $\Delta \cdot A$ is an integer matrix.  

\item  Hypergraphs $C_{r,n,d}$ with the vertex set $I_n^d$ and hyperedges equal to balls of radius $r$ (in the Hamming metric). A transversal  in the hypergraph $C_{r,n,d}$ is a perfect code with distance $2r+1$.

\item Hypergraphs $W_{k,t,n}$ with the vertex set consisting of all $k$-element subsets of an $n$-element set and a hyperedge is a collection of $k$-element subsets containing the same $t$-element set ($t < k <n$). A transversal in the hypergraph $W_{k,t,n}$ is a $t$-$(n,k,1)$-design. 

\item A hypergraph $C_G$ with the vertex set equal to the vertex set of some graph $G$, the hyperedges of $C_G$ are the maximal cliques of $G$. Transversals in  $C_G$ correspond to the maximal independent sets of $G$. 
\end{itemize}

\section*{Acknowledgments}

The results of Sections 3 were funded by the Russian Science Foundation under grant No 22-21-00202. The other work  was carried out within the framework of the state contract of the Sobolev Institute of Mathematics (project no. FWNF-2022-0017).

The author is grateful to Denis Krotov for assistance in implementing Algorithm 1 for enumeration of vertices of $\Omega_4^3$, finding their equivalence classes,  determining  orders of the  automorphism groups and counting permanents.  We also thank  Sergey Vladimirov  for help with the first version of record~\cite{vert34.data}. The author is also grateful to Sergey Avgustinovich, Ivan Mogilnykh, and Anastasia Vasil'eva for discussions during the early stages of this work at a workshop in February 2021.

\section*{Appendix. Vertices of the polytope $\Omega_3^4$}

In this section we list all 21 vertices of $\Omega_3^4$ identified through the implementation of Algorithm~1, as described in Section~\ref{vert43sec}. We  provide a representative for each equivalence class,  calculate the permanent and the size of the support for each vertex, and verify whether the vertex is symmetric.

\#1. The multidimensional permutation (symmetric matrix):

$$
A_1 = \left( 
\begin{array}{ccc|ccc|ccc}
1 & 0 & 0 & 0 & 0 & 1 & 0 & 1 & 0 \\
0 & 0 & 1 & 0 & 1 & 0 & 1 & 0 & 0 \\
0 & 1 & 0 & 1 & 0 & 0 & 0 & 0 & 1 \\
\hline
0 & 0 & 1 & 0 & 1 & 0 & 1 & 0 & 0 \\
0 & 1 & 0 & 1 & 0 & 0 & 0 & 0 & 1 \\
1 & 0 & 0 & 0 & 0 & 1 & 0 & 1 & 0 \\
\hline
0 & 1 & 0 & 1 & 0 & 0 & 0 & 0 & 1 \\
1 & 0 & 0 & 0 & 0 & 1 & 0 & 1 & 0 \\
0 & 0 & 1 & 0 & 1 & 0 & 1 & 0 & 0 \\
\end{array}
\right)
$$
The size of its  support is $N(A_1) = 27$ and the permanent $\per (A_1) = 27$. 

\#2. 
$$
A_2 = \frac{1}{2} \left( 
\begin{array}{ccc|ccc|ccc}
2 & 0 & 0 & 0 & 2 & 0 & 0 & 0 & 2 \\
0 & 1 & 1 & 1 & 0 & 1 & 1 &  1 & 0 \\
0 & 1 & 1 & 1 & 0 & 1 & 1 & 1 & 0 \\
\hline
0 & 1 & 1 & 1 & 0 & 1 & 1 & 1 & 0 \\
2 & 0 & 0 & 0 & 1 & 1  & 0 & 1 & 1 \\
0 & 1 & 1 & 1 & 1 & 0 & 1 & 0 & 1 \\
\hline
0 & 1 & 1  & 1 & 0 & 1 & 1 & 1 & 0 \\
0 & 1 & 1 & 1 & 1 & 0 & 1 & 0 & 1 \\
2 & 0 & 0 & 0 & 1 & 1 & 0 & 1 & 1\\
\end{array}
\right)
$$
The size of its  support is $N(A_2) = 49$ and the permanent $\per (A_2) = 10.5$.

\#3. The dot product of the $3$-dimensional permutation of order $3$ and the vertex $V$ (see Theorem~\ref{dotconstruction}):

$$
A_3 = \frac{1}{2} \left( 
\begin{array}{ccc|ccc|ccc}
2 & 0 & 0 & 0 & 1 & 1 & 0 & 1 & 1 \\
0 & 1 & 1 & 1 & 0 & 1 & 1 &  1 & 0 \\
0 & 1 & 1 & 1 & 1 & 0 & 1 & 0 & 1 \\
\hline
0 & 1 & 1 & 2 & 0 & 0 & 0 & 1 & 1\\
1 & 1 & 0 & 0 & 1 & 1 & 1 & 0 & 1 \\
1 & 0 & 1 & 0 & 1 & 1 & 1 & 1 & 0\\
\hline
0 & 1 & 1 & 0 & 1 & 1 & 2 & 0 & 0 \\
1 & 0 & 1 & 1 & 1 & 0 & 0 & 1 & 1\\
1 & 1 & 0 & 1 & 0 & 1 & 0 & 1 & 1 \\
\end{array}
\right)
$$
The size of its  support is $N(A_3) = 51$ and the permanent $\per (A_3) = 9$. 

\#4.  The symmetric matrix
$$
A_4 = \frac{1}{2} \left( 
\begin{array}{ccc|ccc|ccc}
2 & 0 & 0 & 0 & 1 & 1 & 0 & 1 & 1\\
0 & 1 & 1 & 1 & 0 & 1 & 1 & 1 & 0 \\
0 & 1 & 1 & 1 & 1 & 0 & 1 & 0 & 1 \\
\hline
0 & 1 & 1 & 1& 0 & 1 & 1 & 1 & 0 \\
1 & 0 & 1 & 0 & 2 & 0 & 1 & 0 & 1\\
1 & 1 & 0 & 1 & 0 & 1 & 0 & 1 & 1\\
\hline
0 & 1 & 1 & 1 & 1 & 0 & 1 & 0 & 1 \\
1 & 1 & 0 & 1 & 0 & 1 & 0 & 1 & 1 \\
1 & 0 & 1 & 0 & 1 & 1 & 1& 1 & 0 \\
\end{array}
\right)
$$
The size of its  support is $N(A_4) = 52$ and the permanent $\per (A_4) = 8.25$.  

\#5. 
$$
A_5 = \frac{1}{3} \left( 
\begin{array}{ccc|ccc|ccc}
3 & 0 & 0 & 0 & 1 & 2 & 0 & 2 & 1 \\
0 & 2 & 1 & 2 & 1 & 0 & 1 &  0 & 2 \\
0 & 1 & 2 & 1 & 1 & 1 & 2 & 1 & 0 \\
\hline
0 & 2 & 1 & 2 & 1 & 0 & 1 &  0 & 2 \\
1 & 1 & 1 & 1 & 0 & 2 & 1 & 2 & 0 \\
2 & 0 & 1 & 0 & 2 & 1 & 1 & 1 & 1 \\
\hline
0 & 1 & 2 & 1 & 1 & 1 & 2 & 1 & 0 \\
2 & 0 & 1 & 0 & 2 & 1 & 1 & 1 & 1 \\
1 & 2 & 0 & 2 & 0 & 1 & 0 & 1 & 2 \\
\end{array}
\right)
$$
The size of its  support is $N(A_5) = 58$ and the permanent $\per (A_5) =  \frac{256}{27} \approx 9.48$.

\#6. 
$$
A_6 =  \frac{1}{3} \left( 
\begin{array}{ccc|ccc|ccc}
3 & 0 & 0 & 0 & 2 & 1 & 0 & 1 & 2 \\
0 & 2 & 1 & 1 & 0 & 2 & 2 & 1 & 0 \\
0 & 1 & 2 & 2 & 1 & 0 & 1 & 1 & 1 \\
\hline
0 & 2 & 1 & 2 & 0 & 1 & 1 & 1 & 1 \\
2 & 0 & 1 & 1 & 1 & 1 & 0 & 2 & 1\\
1 & 1 & 1 & 0 & 2 & 1 & 2 & 0 & 1 \\
\hline
0 & 1 & 2 & 1 & 1 & 1 & 2 & 1 & 0 \\
1 & 1 & 1 & 1& 2 & 0 & 1 & 0 & 2\\
2 & 1 & 0 & 1 & 0 & 2 & 0 & 2 & 1 \\
\end{array}
\right)
$$
The size of its  support is $N(A_6) = 59$ and the permanent $\per (A_6) =  \frac{85}{9} \approx 9.44$.

\#7.  The symmetric matrix
$$
A_7 = \frac{1}{3} \left( 
\begin{array}{ccc|ccc|ccc}
3 & 0 & 0 & 0 & 2 & 1 & 0 & 1 & 2 \\
0 & 2 & 1 & 2 & 0 & 1 & 1 &  1 & 1 \\
0 & 1 & 2 & 1 & 1 & 1 & 2 & 1 & 0 \\
\hline
0 & 2 & 1 & 2 & 0 & 1 & 1 &  1 & 1 \\
2 & 0 & 1 & 0 & 1 & 2 & 1 & 2 & 0 \\
1 & 1 & 1 & 1 & 2 & 0 & 1 & 0 & 2 \\
\hline
0 & 1 & 2 & 1 & 1 & 1 & 2 & 1 & 0 \\
1 & 1 & 1 & 1 & 2 & 0 & 1 & 0 & 2 \\
2 & 1 & 0 & 1 & 0 & 2 & 0 & 2 & 1 \\
\end{array}
\right)
$$
The size of its  support is $N(A_7) = 59$ and the permanent $\per (A_7) =  \frac{85}{9} \approx 9.44$.

\#8. 
$$
A_8 =  \frac{1}{4}\left( 
\begin{array}{ccc|ccc|ccc}
4 & 0 & 0 & 0 & 2 & 2 & 0 & 2 & 2 \\
0 & 2 & 2 & 2 & 0 & 2 & 2 &  2 & 0 \\
0 & 2 & 2 & 2 & 2 & 0 & 2 & 0 & 2 \\
\hline
0 & 3 & 1 & 3 & 0 & 1 & 1 &  1 & 2 \\
3 & 0 & 1 & 0 & 2 & 2 & 1 & 2 & 1 \\
1 & 1 & 2 & 1 & 2 & 1 & 2 & 1 & 1 \\
\hline
0 & 1 & 3 & 1 & 2 & 1 & 3 & 1 & 0 \\
1 & 2 & 1 & 2 & 2 & 0 & 1 & 0 & 3 \\
3 & 1 & 0 & 1 & 0 & 3 & 0 & 3 & 1 \\
\end{array}
\right)
$$
The size of its  support is $N(A_8) = 60$ and the permanent $\per (A_8) =  \frac{77}{8} = 9.625$. 

\#9.  The symmetric matrix
$$
A_9 = \frac{1}{4} \left( 
\begin{array}{ccc|ccc|ccc}
0 & 2 & 2 & 2 & 0 & 2 & 2 & 2 & 0 \\
2 & 0 & 2 & 0 & 3 & 1 & 2 &  1 & 1 \\
2 & 2 & 0 & 2 & 1 & 1 & 0 & 1 & 3 \\
\hline
2 & 0 & 2 & 0 & 3 & 1 & 2 &  1 & 1 \\
0 & 3 & 1 & 3 & 1 & 0 & 1 & 0 & 3 \\
2 & 1 & 1 & 1 & 0 & 3 & 1 & 3 & 0 \\
\hline
2 & 2 & 0 & 2 & 1 & 1 & 0 & 1 & 3 \\
2 & 1 & 1 & 1 & 0 & 3 & 1 & 3 & 0 \\
0 & 1 & 3 & 1 & 3 & 0 & 3 & 0 & 1 \\
\end{array}
\right)
$$
The size of its  support is $N(A_9) = 60$ and the permanent $\per (A_9) =  \frac{39}{4} = 9.75$.  

\#10. 
$$
A_{10} =  \frac{1}{3} \left( 
\begin{array}{ccc|ccc|ccc}
1 & 1 & 1 & 1 & 0 & 2 & 1 & 2 & 0 \\
1 & 1 & 1 & 2 & 1 & 0 & 0 &  1 & 2 \\
1 & 1 & 1 & 0 & 2 & 1 & 2 & 0 & 1 \\
\hline
2 & 0 & 1 & 1 & 2 & 0 & 0 &  1 & 2 \\
0 & 1 & 2 & 1 & 1 & 1 & 2 & 1 & 0 \\
1 & 2 & 0 & 1 & 0 & 2 & 1 & 1 & 1 \\
\hline
0 & 2 & 1 & 1 & 1 & 1 & 2 & 0 & 1 \\
2 & 1 & 0  & 0 & 1 & 2 & 1 & 1 & 1 \\
1 & 0 & 2 & 2 & 1 & 0 & 0 & 2 & 1 \\
\end{array}
\right)
$$
The size of its  support is $N(A_{10}) = 61$ and the permanent $\per (A_10) =  \frac{236}{27} \approx 8.74$.  

This vertex of $\Omega_3^4$  was also found in~\cite{AhLoHem.polycones}.

\#11. 
$$
A_{11} = \frac{1}{3} \left( 
\begin{array}{ccc|ccc|ccc}
1 & 1 & 1 & 1 & 0 & 2 & 1 & 2 & 0 \\
1 & 1 & 1 & 2 & 1 & 0 & 0 &  1 & 2 \\
1 & 1 & 1 & 0 & 2 & 1 & 2 & 0 & 1 \\
\hline
1 & 1 & 1 & 0 & 2 & 1 & 2 &  0 & 1 \\
1 & 0 & 2 & 1 & 1 & 1 & 1 & 2 & 0 \\
1 & 2 & 0 & 2 & 0 & 1 & 0 & 1 & 2 \\
\hline
1 & 1 & 1 & 2 & 1 & 0 & 0 & 1 & 2 \\
1 & 2 & 0  & 0 & 1 & 2 & 2 & 0 & 1 \\
1 & 0 & 2 & 1 & 1 & 1 & 1 & 2 & 0 \\
\end{array}
\right)
$$
The size of its  support is $N(A_{11}) = 61$ and the permanent $\per (A_{11}) =  \frac{236}{27} \approx 8.74$. 

\#12. 
$$
A_{12} = \frac{1}{4} \left( 
\begin{array}{ccc|ccc|ccc}
4 & 0 & 0 & 0 & 2 & 2 & 0 & 2 & 2 \\
0 & 2 & 2 & 2 & 0 & 2 & 2 &  2 & 0 \\
0 & 2 & 2 & 2 & 2 & 0 & 2 & 0 & 2 \\
\hline
0 & 1 & 3 & 3 & 1 & 0 & 1 &  2 & 1 \\
3 & 1 & 0 & 0 & 2 & 2 & 1 & 1 & 2 \\
1 & 2 & 1 & 1 & 1 & 2 & 2 & 1 & 1 \\
\hline
0 & 3 & 1 & 1 & 1 & 2 & 3 & 0 & 1 \\
1 & 1 & 2  & 2 & 2 & 0 & 1 & 1 & 2 \\
3 & 0 & 1 & 1 & 1 & 2 & 0 & 3 & 1 \\
\end{array}
\right)
$$
The size of its  support is $N(A_{12}) = 62$ and the permanent $\per (A_{12}) =  \frac{75}{8} =  9.375$. 

\#13.  The symmetric matrix
$$
A_{13} = \frac{1}{5} \left( 
\begin{array}{ccc|ccc|ccc}
5 & 0 & 0 & 0 & 3 & 2 & 0 & 2 & 3 \\
0 & 3 & 2 & 3 & 0 & 2 & 2 &  2 & 1 \\
0 & 2 & 3 & 2 & 2 & 1 & 3 & 1 & 1 \\
\hline
0 & 3 & 2 & 3 & 0 & 2 & 2 &  2 & 1 \\
3 & 0 & 2 & 0 & 2 & 3 & 2 & 3 & 0 \\
2 & 2 & 1 & 2 & 3 & 0 & 1 & 0 & 4 \\
\hline
0 & 2 & 3 & 2 & 2 & 1 & 3 & 1 & 1 \\
2 & 2 & 1  & 2 & 3 & 0 & 1 & 0 & 4 \\
3 & 1 & 1 & 1 & 0 & 4 & 1 & 4 & 0 \\
\end{array}
\right)
$$
The size of its  support is $N(A_{13}) = 62$ and the permanent $\per (A_{13}) =  \frac{1194}{125} =  9.552$. 

\#14. 
$$
A_{14} =  \frac{1}{4} \left( 
\begin{array}{ccc|ccc|ccc}
2 & 1 & 1 & 2 & 0 & 2 & 0 & 3 & 1 \\
1 & 2 & 1 & 2 & 2 & 0 & 1 & 0 & 3 \\
1 & 1 & 2 & 0 & 2 & 2 & 3 & 1 & 0 \\
\hline
2 & 2 & 0 & 1 & 1 & 2 & 1 &  1 & 2 \\
0 & 2 & 2 & 1 & 1 & 2 & 3 & 1 & 0 \\
2 & 0 & 2 & 2 & 2 & 0 & 0 & 2 & 2 \\
\hline
0 & 1 & 3 & 1 & 3 & 0 & 3 & 0 & 1 \\
3 & 0 & 1  & 1 & 1 & 2 & 0 & 3 & 1 \\
1 & 3 & 0 & 2 & 0 & 2 & 1 & 1 & 2 \\
\end{array}
\right)
$$
The size of its  support is $N(A_{14}) = 62$ and the permanent $\per (A_{14}) =  \frac{73}{8} =  9.125$. 

\#15. 
$$
A_{15} = \frac{1}{3} \left( 
\begin{array}{ccc|ccc|ccc}
1 & 1 & 1 & 1 & 0 & 2 & 1 & 2 & 0 \\
1 & 1 & 1 & 0 & 2 & 1 & 2 & 0 & 1 \\
1 & 1 & 1 & 2 & 1 & 0 & 0 & 1 & 2 \\
\hline
1 & 1 & 1 & 1 & 1 & 1 & 1 &  1 & 1 \\
1 & 0 & 2 & 2 & 1 & 0 & 0 & 2 & 1 \\
1 & 2 & 0 & 0 & 1 & 2 & 2 & 0 & 1 \\
\hline
1 & 1 & 1 & 1 & 2 & 0 & 1 & 0 & 2 \\
1 & 2 & 0  & 1 & 0 & 2 & 1 & 1 & 1 \\
1 & 0 & 2 & 1 & 1 & 1 & 1 & 2 & 0 \\
\end{array}
\right)
$$
The size of its  support is $N(A_{15}) = 63$ and the permanent $\per (A_{15}) =  \frac{26}{3} \approx  8.667$. 

\#16. 
$$
A_{16} = \frac{1}{4} \left( 
\begin{array}{ccc|ccc|ccc}
2 & 1 & 1 & 2 & 0 & 2 & 0 & 3 & 1 \\
1 & 2 & 1 & 2 & 2 & 0 & 1 & 0 & 3 \\
1 & 1 & 2 & 0 & 2 & 2 & 3 & 1 & 0 \\
\hline
2 & 0 & 2 & 1 & 3 & 0 & 1 &  1 & 2 \\
0 & 2 & 2 & 1 & 1 & 2 & 3 & 1 & 0 \\
2 & 2 & 0 & 2 & 0 & 2 & 0 & 2 & 2 \\
\hline
0 & 3 & 1 & 1 & 1 & 2 & 3 & 0 & 1 \\
3 & 0 & 1  & 1 & 1 & 2 & 0 & 3 & 1 \\
1 & 1 & 2 & 2 & 2 & 0 & 1 & 1 & 2 \\
\end{array}
\right)
$$
The size of its  support is $N(A_{16}) = 63$ and the permanent $\per (A_{16}) =  9 $. 

\#17. 
$$
A_{17} = \frac{1}{5} \left( 
\begin{array}{ccc|ccc|ccc}
0 & 4 & 1 & 4 & 0 & 1 & 1 & 1 & 3 \\
4 & 0 & 1 & 0 & 3 & 2 & 1 &  2 & 2 \\
1 & 1 & 3 & 1 & 2 & 2 & 3 & 2 & 0 \\
\hline
3 & 1 & 1 & 1 & 2 & 2 & 1 &  2 & 2 \\
1 & 2 & 2 & 2 & 0 & 3 & 2 & 3 & 0 \\
1 & 2 & 2 & 2 & 3 & 0 & 2 & 0 & 3 \\
\hline
2 & 0 & 3 & 0 & 3 & 2 & 3 & 2 & 0 \\
0 & 3 & 2  & 3 & 2 & 0 & 2 & 0 & 3 \\
3 & 2 & 0 & 2 & 0 & 3 & 0 & 3 & 2 \\
\end{array}
\right)
$$
The size of its  support is $N(A_{17}) = 63$ and the permanent $\per (A_{17}) =  \frac{1074}{125} =  8.592$. 

\#18. 
$$
A_{18} = \frac{1}{5} \left( 
\begin{array}{ccc|ccc|ccc}
0 & 4 & 1 & 4 & 0 & 1 & 1 & 1 & 3 \\
4 & 0 & 1 & 0 & 3 & 2 & 1 &  2 & 2 \\
1 & 1 & 3 & 1 & 2 & 2 & 3 & 2 & 0 \\
\hline
4 & 1 & 0 & 0 & 2 & 3 & 1 &  2 & 2 \\
0 & 2 & 3 & 3 & 0 & 2 & 2 & 3 & 0 \\
1 & 2 & 2 & 2 & 3 & 0 & 2 & 0 & 3 \\
\hline
1 & 0 & 4 & 1 & 3 & 1 & 3 & 2 & 0 \\
1 & 3 & 1  & 2 & 2 & 1 & 2 & 0 & 3 \\
3 & 2 & 0 & 2 & 0 & 3 & 0 & 3 & 2 \\
\end{array}
\right)
$$
The size of its  support is $N(A_{18}) = 63$ and the permanent $\per (A_{18}) =  \frac{1141}{125} =  9.128$. 

\#19. The symmetric matrix 
$$
A_{19} = \frac{1}{6} \left( 
\begin{array}{ccc|ccc|ccc}
0 & 5 & 1 & 5 & 0 & 1 & 1 & 1 & 4 \\
5 & 0 & 1 & 0 & 3 & 3 & 1 &  3 & 2 \\
1 & 1 & 4 & 1 & 3 & 2 & 4 & 2 & 0 \\
\hline
5 & 0 & 1 & 0 & 3 & 3 & 1 &  3 & 2 \\
0 & 3 & 3 & 3 & 0 & 3 & 3 & 3 & 0 \\
1 & 3 & 2 & 3 & 3 & 0 & 2 & 0 & 4 \\
\hline
1 & 1 & 4 & 1 & 3 & 2 & 4 & 2 & 0 \\
1 & 3 & 2  & 3 & 3 & 0 & 2 & 0 & 4 \\
4 & 2 & 0 & 2 & 0 & 4 & 0 & 4 & 2 \\
\end{array}
\right)
$$
The size of its  support is $N(A_{19}) = 63$ and the permanent $\per (A_{19}) =  \frac{85}{9} \approx  9.444$.  

\#20. 
$$
A_{20} =  \frac{1}{4}\left( 
\begin{array}{ccc|ccc|ccc}
1 & 2 & 1 & 2 & 1 & 1 & 1 & 1 & 2 \\
2 & 1 & 1 & 1 & 0 & 3 & 1 & 3 & 0 \\
1 & 1 & 2 & 1 & 3 & 0 & 2 & 0 & 2 \\
\hline
1 & 2 & 1 & 0 & 1 & 3 & 3 &  1 & 0 \\
2 & 0 & 2 & 1 & 3 & 0 & 1 & 1 & 2 \\
1 & 2 & 1 & 3 & 0 & 1 & 0 & 2 & 2 \\
\hline
2 & 0 & 2 & 2 & 2 & 0 & 0 & 2 & 2 \\
0 & 3 & 1  & 2 & 1 & 1 & 2 & 0 & 2 \\
2 & 1 & 1 & 0 & 1 & 3 & 2 & 2 & 0 \\
\end{array}
\right)
$$
The size of its  support is $N(A_{20}) = 64$ and the permanent $\per (A_{20}) =  \frac{145}{16} = 9.0625$. 

\#21. 
$$
A_{21} = \frac{1}{5}\left( 
\begin{array}{ccc|ccc|ccc}
1 & 2 & 2 & 3 & 1 & 1 & 1 & 2 & 2 \\
2 & 1 & 2 & 1 & 4 & 0 & 2 &  0 & 3 \\
2 & 2 & 1 & 1 & 0 & 4 & 2 & 3 & 0 \\
\hline
2 & 3 & 0 & 0 & 2 & 3 & 3 &  0 & 2 \\
0 & 2 & 3 & 4 & 0 & 1 & 1 & 3 & 1 \\
3 & 0 & 2 & 1 & 3 & 1 & 1 & 2 & 2 \\
\hline
2 & 0 & 3 & 2 & 2 & 1 & 1 & 2 & 1 \\
3 & 2 & 0  & 0 & 1 & 4 & 2 & 2 & 1 \\
0 & 3 & 2 & 3 & 2 & 0 & 2 & 0 & 3 \\
\end{array}
\right)
$$
The size of its  support is $N(A_{21}) = 65$ and the permanent $\per (A_{21}) =  \frac{223}{25} =  8.92$.

\end{document}